\documentclass[preprint,12pt]{elsarticle}
\usepackage{amsmath,amssymb,amsthm,graphicx,url,xcolor}

\newtheorem{thr}{Theorem}
\newtheorem{lem}[thr]{Lemma}
\newtheorem{cor}[thr]{Corollary}

\theoremstyle{definition}
\newtheorem{defn}[thr]{Definition}

\newtheorem{prob}[thr]{Problem}
\newtheorem{compprob}[thr]{Problem}

\theoremstyle{remark}

\journal{arXiv.org}

\def\R{\mathbb{R}}

\begin{document}

\begin{frontmatter}

\title{On tropical and nonnegative \\ factorization ranks of band matrices}

\author{Yaroslav Shitov}
\address{National Research University --- Higher School of Economics, 20 Myasnitskaya ulitsa, 101000, Moscow, Russia}
\ead{yaroslav-shitov@yandex.ru}

\begin{abstract}
Matrix factorization problems over various semirings naturally arise in different contexts of modern pure and applied mathematics. These problems are very hard in general and cause computational difficulties in applications. We give a survey of what is known on the algorithmic complexity of Boolean, fuzzy, tropical, nonnegative, and positive semidefinite factorizations, and we examine the behavior of the corresponding rank functions on matrices of bounded bandwidth. 
We show that the Boolean, fuzzy, and tropical versions of matrix factorization become polynomial time solvable when restricted to this class of matrices, and we also show that the nonnegative rank of a tridiagonal matrix is easy to compute. We recall several open problems from earlier papers on the topic and formulate many new problems.
\end{abstract}

\end{frontmatter}

\section{Introduction}

The problem of matrix factorization over semirings, which is a generalization of the classical rank decomposition problem, has numerous applications in different contexts of mathematics. A notable particular case of this problem is the \textit{nonnegative matrix factorization} (NMF), which provides an important tool for studying various questions both in pure and applied mathematics. In particular, the ability of NMF to extract sparse and meaningful features from nonnegative data vectors leads to important techniques in image processing and text mining~\cite{LS}, computational biology~\cite{Dev}, clustering~\cite{DHS}, and many other problems related to the real-world applications (see~\cite{Gillis} for an extensive survey). A geometrical view on NMF leads to a certain problem with nested polytopes (see~\cite{Yan}), which turns out to be useful for studying the so called \textit{extended formulations} of polytopes and relaxations of linear programming versions of hard problems in combinatorial optimization (see~\cite{FMPTdW}). Several other fields where NMF is useful are statistics~\cite{KRS}, communication complexity~\cite{LS2}, and quantum mechanics~\cite{CR}.

The tropical version of matrix factorization does also have applications in optimization. As shown in~\cite{Barv4}, the Traveling Salesman problem (TSP) can be solved in polynomial time if the corresponding distance matrix is given as a tropical product of an $n\times k$ and $k\times n$ matrices, where $k$ is a constant fixed in advance. (This version of TSP has a natural interpretation and became known as \textit{TSP with $k$ warehouses}, see~\cite{Barv3}.) The paper~\cite{Barv2} presents further examples of hard optimization problems which become solvable in polynomial-time when reduced to matrices represented in the way mentioned above. From point of view of tropical geometry, the factorization rank can be described as the minimum number of points whose tropical convex hull contains the columns of a given matrix, and this characterization may be of independent interest, see~\cite{Dev2, DSS}.

The Boolean version of matrix factorization gives rise to the concept known as \textit{Boolean rank}, which serves as a lower bound for both nonnegative and tropical factorization ranks, see~\cite{FKPT, mycomb}. The Boolean rank is also of independent interest in graph theory, where it arises as the \textit{bipartite dimension} for bipartite graphs, see~\cite{FH}. Another related concept is the \textit{fuzzy rank} of a matrix, which is the factorization rank of a matrix over the fuzzy algebra, and it also valuable in different applications, see~\cite{Cho}.

\section{Matrix factorization and its complexity}

Formally speaking, a \textit{semiring} is a set $R$ equipped with two binary operations $\oplus$, $\odot$ and two distinguished elements $0$, $1$ such that (i) $(R,\oplus,0)$ is a commutative monoid, (ii) $(R,\odot,1)$ is a monoid, (iii) multiplication distributes over addition from both sides, and (iv) one has $0\odot x=x\odot 0=0$ for any $x\in R$. In other words, the elements of a semiring are not required to have additive inverses (and if this is the case, the corresponding semiring becomes a ring). Let us mention the particular examples of semirings relevant in our study:

\noindent the \textit{nonnegative reals} $(\R_+,+,\cdot)$;

\noindent the \textit{tropical semiring} $(\R_+,\max,\cdot)$;

\noindent the \textit{fuzzy algebra} $([0,1],\max,\min)$;

\noindent the \textit{binary Boolean semiring} $(\{0,1\},\max,\cdot)$.

In all of these cases, the conventional zero and one serve as additive and multiplicative identity elements, respectively. (We note that the tropical semiring is usually defined as $(\R\cup\{\infty\},\min,+)$ but we decided to use the isomorphic structure corresponding to the multiplicative notation in order to make the relations to other semirings more explicit, see also~\cite{mysqr, TH}.)

The problem of matrix factorization aims to represent a given matrix as a product of two matrices with a fixed (usually relatively small) inner dimension. We recall that the multiplication of matrices over semirings is defined in the same way as ordinary matrix multiplication but with $+$ and $\cdot$ replaced by the corresponding arithmetic operations $\oplus$ and $\odot$. A concept related to the matrix factorization problem is the factorization rank.

\begin{defn}\label{factrank}
The \textit{factorization rank} of a matrix $A$ over a semiring $S$ is the smallest integer $k$ for which there exist matrices $B\in S^{m\times k}$ and $C\in S^{k\times n}$ satisfying $B\odot C=A$. By convention, the factorization rank of a zero matrix is assumed to be zero.
\end{defn}

In general, factorization ranks are hard to compute. For instance, both the tropical and fuzzy versions of matrix factorization contain the Boolean version as a special case, which is in turn equivalent to the so called \textsc{set basis} problem. The latter one was proved to be NP-complete by Orlin in~\cite{Orlin}, so the above mentioned factorization problems are NP-hard. (As can be seen from~\cite{myfactor}, they fall in NP when restricted to rational matrices, and the corresponding restrictions become NP-complete.) In the above mentioned paper Orlin did also conjecture that the related \textsc{bicontent} problem, also known as \textsc{normal set basis}, is NP-complete, and this fact was proven later (the earliest reference I am aware of is Lemma~3.3 in the paper~\cite{JR} by Jiang and Ravikumar). In a fashion similar to how \textsc{set basis} gives the NP-hardness of the Boolean rank, the result of Jiang and Ravikumar proves the NP-hardness of the nonnegative matrix factorization (see the discussion in~\cite{myCR}). However, it turned out that the NMF problem is unlikely to be NP-complete, as shows the paper~\cite{myuni} determining the true complexity of NMF. Namely, the NMF problem is $\exists\R$-complete, that is, polynomial time equivalent to the determination whether a given system of polynomial equations has a real solution or not. Let us also mention a notable paper by Vavasis~\cite{Vav}, who proves that it is NP-hard to test whether the nonnegative rank of a matrix equals its conventional rank; the $\exists\R$-completeness of this problem is still an open issue.

\begin{prob}
Is it $\exists\R$-complete to check whether the nonnegative rank of a given rational nonnegative matrix equals its conventional rank?
\end{prob}

A special case of the factorization problem that is particularly interesting for applications is when the inner dimension of the factors is small or even bounded above by a constant fixed in advance. This restriction makes the Boolean rank problem easy as it turns to be what is called \textit{fixed parameter tractable}. Namely, it is a standard exercise in parametrized complexity theory that $m\times n$ matrices of Boolean rank $r$ can be detected in time $O(f(r)(m+n)^\theta)$, where $f(r)$ is a computable function depending only on $r$, and $\theta$ is an absolute constant, see Exersice~1.42 in~\cite{pct}. The NMF problem does also become easier when restricted to fixed parameters---for any integer $k$, the paper~\cite{AGKM} gives a polynomial time algorithm detecting matrices of nonnegative rank $k$. However, the exponent of the size of a given matrix in the estimate of the time complexity of the algorithm in~\cite{AGKM} depends on $k$, and the fixed parameter tractability of nonnegative rank remains an open issue.

\begin{prob}\label{nnfpt}
Is nonnegative rank fixed-parameter tractable? In other words, do there exist a function $f:\mathbb{Z}\to\mathbb{Z}$, a constant $\theta\in\mathbb{R}$, and an algorithm that halts in time $f(k)\Lambda^{\theta}$ and decides whether the nonnegative rank of a rational matrix $A$ of total bit length $\Lambda$ is at most $k$?
\end{prob}

There is another concept, known as \textit{positive semidefinite rank}, which is closely related to nonnegative rank. We recall that the positive semidefinite rank of a nonnegative matrix $A$ is the smallest integer $k$ for which there exist $k\times k$ real positive semidefinite matrices $B_1,\ldots,B_m$, $C_1,\ldots,C_n$ satisfying $A_{ij}=\operatorname{tr} (B_iC_j)$ for all $i,j$. The situation with the algorithmic complexity of this invariant is similar to the one for nonnegative rank. Namely, the positive semidefinite rank is $\exists\mathbb{R}$-complete in general but can be computed in polynomial time if given matrices are promised to have ranks bounded by a number fixed in advance, see~\cite{mypsd, mypsd2}. Similarly to Problem~\ref{nnfpt}, the fixed-parameter tractability of positive semidefinite rank is an open question.

The situation with tropical factorization rank is opposite to the one with Boolean and nonnegative ranks. According to the paper~\cite{myfactor}, the problem of detecting matrices with tropical factorization rank equal to any integer $\geqslant7$ fixed in advance is NP-complete. On the other hand, there are polynomial time algorithms detecting matrices with tropical factorization ranks $1, 2, 3$, see~\cite{mycomb}. The author believes that the same problem for rank $4$ is still polynomial time solvable, but he has no idea as to the true complexity of it for ranks $5$ and $6$. Also, the author is not aware of any result on the parametrized complexity of fuzzy ranks.

\begin{prob}
Is fuzzy rank fixed-parameter tractable? Is there, for any integer $k$ fixed in advance, a polynomial time algorithm detecting matrices of fuzzy rank $k$? How hard is it to detect matrices of fuzzy rank three?
\end{prob}

In this paper, we are going to examine the complexity of these factorization problems for \textit{band} matrices, that is, square matrices whose $(i,j)$ entries are zero except when $|i-j|$ is small (namely, bounded by a number fixed in advance). The classes of such matrices include tridiagonal and pentadiagonal matrices, and they often occur in numerical analysis and finite element problems in particular. Many standard linear algebraic procedures work faster when applied to such matrices, and it may be interesting to look at what happens with the factorization ranks in this case. In Section~3, we show that the Boolean, tropical, and fuzzy factorization ranks of band matrices can be found in polynomial time. In Section~4, we examine the complexity of nonnegative rank, but we managed to construct the polynomial time algorithm for tridiagonal matrices only. 
We point out several directions of further research and open problems; in particular, the algorithmic complexity of nonnegative rank for general matrices of bounded bandwidth remains unknown.

\section{Fast fuzzy and tropical factorization of band matrices}

In order to construct a polynomial time algorithm for tropical and fuzzy factorization ranks of band matrices, we need to recall several facts from graph theory. We say that a square matrix $M$ is $k$\textit{-band} if $M_{ij}=0$ whenever $|i-j|>k$. The \textit{bandwidth} of a simple graph $G=(V,E)$ is the minimal integer $k$ for which there exists a labeling $\varphi:V\to\mathbb{Z}$ such that, for any edge $\{u,v\}\in E$, one has $|\varphi(u)-\varphi(v)|\leqslant k$. In other words, the bandwidth is the smallest integer $k$ for which the adjacency matrix of a graph can be made $k$-band by permutation similarities. The following problem will be extensively used for reductions in the considerations of this section. 

\begin{compprob}\label{prhs}(HITTING SET.)

\noindent Given: A bipartite graph $G=(V_1,V_2,E)$ and an integer $k$.

\noindent Question: Are there $k$ vertices in $V_1$ that \textit{dominate} $V_2$? (That is, every vertex in $V_2$ should be adjacent to at least one of the $k$ chosen vertices.)
\end{compprob}

\begin{thr}\label{thrhsp}(See the last sentence of Section~4 in~\cite{MS}.)
For any fixed $b$, the restriction of \textsc{hitting set} to graphs of bandwidth at most $b$ is in P.
\end{thr}

Let $M\in\R^{I\times J}$ be a matrix; the subset $S(M)\subset I\times J$ of all non-zero positions in $M$ is to be called the \textit{support} of $M$. We will write $M_1\leqslant M_2$ if matrices $M_1, M_2$ have equal sizes and each entry of $M_1$ does not exceed the corresponding entry of $M_2$. We say that a subset $\alpha\subset S(M)$ is $t$-admissible (or $f$-admissible, respectively) if there is a matrix $Q$ which has tropical (or fuzzy, respectively) factorization rank one, satisfies $Q\leqslant M$, and the equality $Q_{ij}=M_{ij}$ holds if and only if $(i,j)\in\alpha$.

\begin{lem}\label{lemadm}
Let $M\in\R^{n\times n}$ be a $k$-band matrix, $i,j\in\{1,\ldots,n\}$. All the $f$-admissible and $t$-admissible subsets of $M$ that contain $(i,j)$ can be determined within the number of arithmetic operations bounded by a function of $k$. Also, the differences between indexes of rows (columns) for any pair of entries in one such subset are bounded by a function of $k$.
\end{lem}

\begin{proof}
Let $Q$ be a rank-one matrix satisfying $Q\leqslant M$. This implies $S(Q)=I'\times J'$ and $S(Q)\subset S(M)$, so since $M$ is a $k$-band matrix, we get $I'\subset[i-2k,i+2k]$ and $J'\subset[j-2k,j+2k]$. In other words, the problem reduces to a $(4k+1)\times(4k+1)$ submatrix of $M$, and it can be solved by the quantifier elimination techniques (see~\cite{Reneg3}) within the number of arithmetic operations bounded by a function of $k$.
\end{proof}

\begin{thr}\label{thrfast}
Let $k$ be a fixed integer. The fuzzy and tropical factorization ranks of a $k$-band $n\times n$ matrix can be computed in polynomial time.
\end{thr}

\begin{proof}
The factorization rank is equal to the smallest $r$ for which a matrix can be written as a sum of $r$ matrices that have rank one with respect to a corresponding semiring. Since the sum of fuzzy or tropical matrices is their entrywise maximum, the factorization rank is equal to the smallest number of corresponding admissible sets covering all the non-zero entries of the matrix.

Now we describe a bipartite graph $G$ which is a reduction to Problem~\ref{prhs}. The $V_1$ part of vertices are the admissible sets computed as in Lemma~\ref{lemadm}, the $V_2$ part are the entries in the support of the matrix, and the pair of an entry and an admissible set is an edge if and only if the former belongs to the latter. Then we label the elements of $V_2=\{u_1,\ldots,u_m\}$ in lexicographical order and define the sequence
$$u_1,w_{11},\ldots,w_{c_11},u_2,w_{21},\ldots,w_{2c_2},\ldots,u_m,w_{m1},\ldots,w_{mc_m},$$
where $w_{it}$'s are the admissible sets containing $u_i$ taken in arbitrary order.
Now we define the function $\varphi:V_1\cup V_2\to\mathbb{N}$ sending a vertex of $G$ to the number of its first appearance in the above sequence. According to Lemma~\ref{lemadm}, $\{u_i,w_{jt}\}$ can be an edge only if $|\varphi(u_i)-\varphi(w_{jk})|$ is bounded by a function depending only on $k$. The application of Theorem~\ref{thrhsp} concludes the proof.
\end{proof}


\section{Nonnegative ranks of tridiagonal matrices}

Unfortunately, we cannot prove a complete analogue of Theorem~\ref{thrfast} for nonnegative and positive semidefinite ranks. In this section, we construct a fast algorithm computing the nonnegative rank of a tridiagonal matrix. We say that a matrix is a \textit{matrix unit} if it has exactly one non-zero entry; by $\operatorname{rank}_+(A)$ we denote the nonnegative rank of a nonnegative matrix $A$.

\begin{lem}\label{lemremove}
Define a block matrix
$$A=\left(\begin{array}{c|c}
B&C\\\hline
O&D
\end{array}\right),$$
where $B,C,D$ are nonnegative matrices, and $O$ is a zero matrix. Then $\operatorname{rank}_+(A)\geqslant \operatorname{rank}_+(B)+\operatorname{rank}_+(D)$. Further,

\noindent (1) $\operatorname{rank}_+(B|C)=\operatorname{rank}_+(B)$ implies $\operatorname{rank}_+(A)=\operatorname{rank}_+(B)+\operatorname{rank}_+(D)$, and

\noindent (2) if $C$ is a matrix unit and $\operatorname{rank}_+(B|C)=\operatorname{rank}_+(B)+1$, then $\operatorname{rank}_+(A)= \operatorname{rank}_+(B)+\operatorname{rank}_+\left(\begin{array}{c}
C\\\hline
D
\end{array}\right)$.
\end{lem}

\begin{proof}
Let $\operatorname{rank}_+(A)=r$, and let $A=A_1+\ldots+A_r$ be a representation of $A$ as a sum of nonnegative rank one matrices. Clearly, no $A_i$ can have non-zero entries in both the $B$ and $D$ parts, which implies $\operatorname{rank}_+(A)\geqslant \operatorname{rank}_+(B)+\operatorname{rank}_+(D)$, and the item (1) does also follow. The condition as in item (2) allows us to assume without loss of generality that no $A_i$ has non-zero entries in the $B$ and $C$ parts at the same time, which implies the desired conclusion.
\end{proof}

\begin{cor}\label{cor0}
Let $A$ be a nonnegative matrix in which an $(i,j)$ entry is positive and all the other entries in the $j$th column are zero. Then the nonnegative rank of $A$ equals one plus the nonnegative rank of the matrix obtained from $A$ by removing the $i$th row and $j$th column.
\end{cor}

In what follows, we say that the entries at the $(i,i+1)$ positions of an $n\times n$ matrix form its \textit{superdiagonal}, and the entries $(i+1,i)$ are its \textit{subdiagonal}. The following is a special case of the main result of this section.

\begin{lem}\label{lemnzd}
Let $A$ be an $n\times n$ nonnegative matrix whose subdiagonal and superdiagonal entries are all non-zero. Then the nonnegative rank of $A$ is either $n-1$ or $n$, and we can decide which is the case within $O(n)$ arithmetic operations.
\end{lem}

\begin{proof}
Since the subdiagonal is non-zero, the rank of $A$ is at least $n-1$, so $\operatorname{rank}_+(A)$ can be either $n-1$ or $n$.
If $A_{11}=0$, then we use Corollary~\ref{cor0} and conclude that $\operatorname{rank}_+(A)=2+\operatorname{rank}_+(A')$, where $A'$ is the matrix obtained from $A$ by removing the first two rows and columns, and its nonnegative rank can be computed by the recursive call of the algorithm being constructed.

Now assume $A_{11}>0$. We are going to look for a decomposition $A=A_1+\ldots+A_{n-1}$, where the $A_i$'s are nonnegative rank one matrices. Since $A_{31}=\ldots=A_{n1}=0$, any $A_i$ that has a non-zero in the first column must have zeros in rows indexed $3,\ldots,n$. Since the matrix formed by these rows has rank $n-2$, we see that at most one $A_i$ can have non-zeros in the first column, and, by symmetry, at most one $A_i$ can have non-zeros in the first row. Therefore, there should be an $A_i$ whose first row and column coincide with those of $A$; we have $\operatorname{rank}_+(A)=n-1$ if and only if $A-A_i$ is a nonnegative matrix which satisfies $\operatorname{rank}_+(A-A_i)=n-2$. It remains to check the latter condition with another recursive call.
\end{proof}

\begin{thr}
The nonnegative rank of a tridiagonal $n\times n$ matrix can be computed within $O(n)$ arithmetic operations.
\end{thr}

\begin{proof}
Such a matrix can be presented in block-diagonal form as
$$A=\left(\begin{array}{c|c|c|c|c}
D_1&U_1&O&\ldots&O\\\hline
V_1&D_2&U_2&\ddots&\vdots\\\hline
O&V_2&\ddots&\ddots&O\\\hline
\vdots&\ddots&\ddots&D_{k-1}&U_{k-1}\\\hline
O&\ldots&O&V_{k-1}&D_k
\end{array}\right),$$
where the $D_i$'s are tridiagonal $n_i\times n_i$ matrices with non-zero subdiagonals and non-zero superdiagonals, the $O$'s are zero blocks, every $U_i$ (respectively, $V_j$) is either a zero matrix or a bottom-left (respectively, top-right) matrix unit, and, for every $i$, either $U_i$ or $V_i$ is a zero matrix.

Assuming that $V_1$ is zero (otherwise we consider $A^\top$ instead of $A$), we employ the algorithm as in Lemma~\ref{lemnzd} to the $D_1$ block. If this block has full nonnegative rank, that is, $\operatorname{rank}_+(D_1)=n_1$, then we use the item (1) of Lemma~\ref{lemremove} and get $\operatorname{rank}_+(A)=n_1+\operatorname{rank}_+(A')$, where $A'$ is the matrix obtained from $A$ by removing the first $n_1$ rows and $n_1$ columns. In particular, we can compute $\operatorname{rank}_+(A')$ recursively, so we are done.

If $\operatorname{rank}_+(D_1)=n_1-1$, then we are under the assumptions of the item (2) of Lemma~\ref{lemremove}. We get
$\operatorname{rank}_+(A)=n_1-1+\operatorname{rank}_+(A'')$,
where $A''$ is the matrix obtained from $A$ by removing the first $n_1$ columns. Denoting by $t$ the smallest index for which $U_t$ is zero (or taking $t=k$ if all the $U_i$'s are zero), we apply Corollary~\ref{cor0} and get
$$\operatorname{rank}_+(A)=n_1+\ldots+n_t-1+\operatorname{rank}_+\left(\overline{A}\right),$$
where $\overline{A}$ is the matrix obtained from $A$ by removing the first $n_1+\ldots+n_t$ rows and $n_1+\ldots+n_t$ columns. Finally, we compute $\operatorname{rank}_+\left(\overline{A}\right)$ recursively.
\end{proof}

For $k\geqslant2$, we do not know if the nonnegative rank of a $k$-band matrix can be computed in polynomial time. In fact, the positive semidefinite rank looks even harder to deal with, and we did not manage to compute it even for several specific tridiagonal matrices.

\begin{prob}\label{probnonpen}
Is there a polynomial time algorithm computing the nonnegative rank of a pentadiagonal matrix?
\end{prob}

\begin{prob}\label{probpsdtri}
Is there a polynomial time algorithm computing the positive semidefinite rank of a tridiagonal matrix?
\end{prob}

\begin{prob}\label{probpsdtoep}
What is the positive semidefinite rank of the $n\times n$ matrix $A$ with ones on the diagonal and superdiagonal, with $a$'s on the subdiagonal, and with zeros everywhere else? 
\end{prob}

Problem~\ref{probpsdtri} can be related to the question asked in~\cite{psdrank}. Namely, is it NP-hard to decide if the positive semidefinite rank of a $n\times n$ matrix equals $n$? Because of the lack of efficient methods for computing this function, a consideration of particular instances like that in Problem~\ref{probpsdtoep} may be helpful in studying these questions.

\section{Acknowldgement}

This research was financially supported by the Russian Science Foundation (project No. 17-71-10229).

\end{document}